\documentclass{amsart}
\usepackage{all2015}
\usepackage{pxfonts}
\newcommand{\ZZo}{\ZZ_{\geq 0}}

\newcommand{\Sinf}{\mathrm{S}_\infty}
\newcommand{\Sgn}{\mathrm{S}_{>n}}
\newcommand{\Sn}{\mathrm{S}_{n}}
\newcommand{\Syz}{\operatorname{Syz}}
\usepackage{algpseudocode}
\usepackage{algorithm}

\begin{document}

\title{Stillman's conjecture via generic initial ideals}

\author{Jan Draisma and Micha{\l} Laso{\'n} and Anton Leykin}

\address[Jan Draisma]{Mathematical Institute, University of
Bern,
Sidlerstrasse 5, 3012 Bern, Switzerland; and Eindhoven
University of Technology, The Netherlands}
\thanks{JD is partially supported by a Vici grant from the
Netherlands Organisation for Scientific Research (NWO)}
\email{jan.draisma@math.unibe.ch}

\address[Micha{\l} Laso{\'n}]{Mathematical Institute,
University of Bern, Sidlerstrasse 5, 3012 Bern,
Switzerland; Institute of Mathematics of the Polish Academy of Sciences, ul.\'{S}niadeckich 8, 00-656 Warszawa, Poland} 
\email{michalason@gmail.com}

\address[Anton Leykin]{School of Mathematics, 
Georgia Institute of Technology, Atlanta, GA 30332-0160,
USA}
\email{leykin@math.gatech.edu}
\thanks{Research of AL is supported in part by DMS-1719968 award from NSF}

\begin{abstract}
Using recent work by Erman-Sam-Snowden, we show that finitely generated
ideals in the ring of bounded-degree formal power series in
infinitely many variables have finitely
generated Gr\"obner bases relative to the graded reverse lexicographic
order. We then combine this result with the first author's work on
topological Noetherianity of polynomial functors to give an algorithmic
proof of the following statement: ideals in polynomial rings generated
by a fixed number of homogeneous polynomials of fixed degrees only
have a finite number of possible generic initial ideals, independently
of the number of variables that they involve and independently of the
characteristic of the ground field. 
Our algorithm outputs not only a finite list of possible generic initial ideals, 
but also finite descriptions of the corresponding strata in
the space of coefficients.
\end{abstract}

\maketitle

\section{Introduction}

\subsection*{Grevlex series and Gr\"obner bases}

Let $A$ be a ring and let $R_A$ be the $A$-algebra of formal power
series over $A$ of bounded degree in the infinitely many variables
$x_1,x_2,\ldots$. In other words, each element of $R_A$ is a formal
infinite sum
\[ \sum_{\alpha \in \NN^{\ZZo}, |\alpha| \leq d} c_\alpha x^\alpha \] 
where $d$ is some nonnegative integer and $c_\alpha \in A$ for each
sequence $\alpha=(\alpha_1,\alpha_2,\ldots)$ of nonnegative integers whose
sum $|\alpha|$ is (finite and) at most $d$. Addition and multiplication
are as usual. 

We equip the polynomial ring $R_A$ with the graded reverse lexicographic order
grevlex, in which $x^\alpha > x^\beta$ if either $|\alpha|>|\beta|$
or $|\alpha|=|\beta|$ and the last non-zero entry of $\alpha-\beta$
is negative. So, for instance, the monomials of degree $3$ are ordered
as follows:
\[
x_1^3>x_1^2x_2>x_1x_2^2>x_2^3>x_1^2x_3>x_1x_2x_3>x_2^2x_3>
x_1x_3^2>x_2x_3^2>x_3^3>x_1^2x_4>\ldots
\]
To remind the reader that this is the only monomial order considered
in this paper, we call the elements of $R_A$ {\em grevlex series} over
$A$. If $f$ is a nonzero element of $R$, then $\lmon(f)$ denotes the
largest monomial that has a nonzero coefficient in $f$,
$\lc(f)$ denotes that coefficient, and
$\lt(f)=\lc(f)\lmon(f)$ is the leading term. The ring $R_A$
carries a unique topology in which a basis of open neighborhoods of
$f \in R_A$ is given by all sets $\{g \in R_A \mid \lmon(f-g)
< x^\alpha\}$ as $\alpha$ varies.

Let $L$ be a field. A {\em Gr\"obner basis} of an ideal $I \subseteq
R_L$ is a subset $B \subseteq I$ such that for each $h \in L$ there
exists an $f \in B$ with $\lmon(f)|\lmon(h)$. We do not require that $B$
be finite. As in the classical setting, a Gr\"obner basis $B$ of $I$
generates $I$ as an ideal (Lemma~\ref{lm:Division}). Our first main
result is the following.

\begin{thm} \label{thm:GBExists}
For every field $L$, every finitely generated homogeneous ideal in
the ring $R_L$ has a finite Gr\"obner basis with respect to grevlex.
\end{thm}

The analogous statement certainly does not hold for all monomial
orders: in \cite[Appendix A.2]{Snellman98} it is shown that the
ideal generated by a generic quadric and a generic cubic has a
non-finitely generated initial ideal relative to the lexicographic
order.  Theorem~\ref{thm:GBExists} implies a positive answer to
\cite[Question 7.1]{Snellman98b}; in that paper a positive answer
is given in the case where the ideal is generated by series 
$\sum_{|\alpha|=d_i} c_{i,\alpha} x^{\alpha},\ i=1,\ldots,k$ whose
coefficients $(c_{i,\alpha})_{i \in [k],|\alpha|=d_i}$ are algebraically
independent over the prime field of $L$.

The natural question arises whether a Gr\"obner basis as in the
theorem can be computed in finite time. A straightforward variant {\sc
SeriesBuchberger} of Buchberger's algorithm shows that this would,
indeed, be the case---{\em if only we could work effectively with
infinite series}.

Next we focus on the following setting where we can indeed work with such
series. Let $\Sinf=\bigcup_n \Sn$ be the union of all symmetric groups,
and let $\Sgn$ be the subgroup of all permutations fixing $1,\ldots,n$
elementwise. Suppose that we are given an action of $\mathrm{S}_{>n_0}$
on $A$ by means of ring automorphisms, and let $\mathrm{S}_{>n_0}$
act on the variables $x_1,x_2,\ldots$ via $\pi x_i=x_{\pi(i)}$. This
action extends to an action of $\mathrm{S}_{>n_0}$ by (continuous)
ring automorphisms on $R_A$ via
\[ \pi \left(\sum_\alpha c_\alpha x^\alpha \right) = \pi \left(\sum_\alpha c_\alpha \prod_i x^{\alpha_i}_i \right)=\sum_{\alpha}
\pi(c_\alpha) \prod_i x^{\alpha_i}_{\pi(i)} = \sum_\alpha \pi(c_\alpha)
x^{\alpha \circ \pi^{-1}}. \]
We call an $f \in R_A$ {\em eventually invariant} if there exists an
$n \geq n_0$ such that $\pi(f)=f$ for all $\pi \in \Sgn$. To specify
an eventually invariant grevlex series we need only a finite number of
coefficients: if $f$ is invariant under $\Sgn$ and has degree $d$, then
$\Sgn$ has only finitely many orbits on monomials in $x_1,x_2,\ldots$
of degree at most $d$---the grevlex-largest element in each orbit
is of the form $x^\alpha$ where $\alpha(n+1) \geq \alpha(n+2) \geq
\ldots$. Then $f$ is uniquely determined by its coefficients on these
grevlex-largest representatives $x^{\alpha_1},\ldots,x^{\alpha_s}$.
We call $\hat{f}:=\sum_{i=1}^s c_{\alpha_i} x^{\alpha_i}$ the {\em
$n$-representation} of $f=\sum_\alpha c_\alpha x^{\alpha}$.
Often we will suppress $n$ from this notation. 

\begin{thm} \label{thm:GBComputable}
Suppose that $A=L$ is a field.  There exist a finite algorithm that on
input a finite list $\hat{f}_1,\ldots,\hat{f}_k$ of representations
of eventually invariant grevlex series $f_1,\ldots,f_k$ outputs a
finite list $\hat{g}_1,\ldots,\hat{g}_l$ representing an eventually
invariant Gr\"obner basis $g_1,\ldots,g_l$ of $\langle f_1,\ldots,f_k
\rangle_{R_L}$.
\end{thm}

\subsection*{Stillman's conjecture}

The condition of eventual invariance seems rather restrictive, but it
is tailored to a proof of the following theorem.

\begin{thm} \label{thm:Stillman}
There exists a finite algorithm that on input $k \in \ZZ$ and
$d_1,\ldots,d_k \in \ZZo$ outputs a finite sequence $S_1,\ldots,S_t$,
each $S_i$ a finite set of monomials in the $x_j$, such that the
following holds: For every infinite field $K$, all $n \in \NN$, and
all homogeneous polynomials $f_1,\ldots,f_k \in K[x_1,\ldots,x_n]$
of degrees $d_1,\ldots,d_k$, respectively, the generic grevlex initial
ideal of $\langle f_1,\ldots,f_k \rangle_{K[x_1,\ldots,x_n]}$ equals
$\langle S_i \rangle_{K[x_1,\ldots,x_n]}$ for some $i$.
\end{thm}

In short: ideals in polynomial rings generated by homogeneous
polynomials of degrees $d_1,\ldots,d_k$ have only finitely many possible
generic grevlex initial ideals, independently of the number of variables.
Via \cite[Corollary 19.11]{Eisenbud95}, which is based on \cite{Bayer87},
this implies that the projective dimension of an ideal generated by
homogeneous forms of fixed degrees but in an arbitrary number of variables
and in arbitrary characteristic is uniformly bounded. This is Stillman's
conjecture from the title; see \cite{Peeva09}.

This is the fourth proof of Stillman's conjecture, after the
first proof by Ananyan-Hochster \cite{Hochster16} and two recent
proofs by Erman-Sam-Snowden \cite{Erman18}. Our proof is the same
in spirit as the second proof in the latter paper in that it uses
Draisma's theorem on topological Noetherianity of polynomial functors
\cite{Draisma17}. However, unlike the second proof in \cite{Erman18}
(but like the first proof there, and like Ananyan-Hochster's proof), our
theorem yields $S_1,\ldots,S_t$ that are valid in all characteristics.
Also, our theorem is constructive in the sense that we give an algorithm
for computing the possible initial ideals and the corresponding strata
given by equations and disequations for field characteristics and
coefficients of the input series.  All these are represented finitely.

In \cite{Erman18} the authors raise the question whether a version over
$\ZZ$ of Draisma's theorem holds, as this would also make their second
proof characteristic-inde\-pen\-dent. We do not settle this question. Instead,
the algorithm of Theorem~\ref{thm:Stillman} simulates a generic ideal
computation in all characteristics, branching along constructible
subsets of $\Spec \ZZ$ whenever necessary. We argue that, if there were
an infinite branch in this computation, then this branch would also
be infinite over some field; and that this would contradict Draisma's
theorem over that field.

In \cite{Erman17} (see also \cite[Theorem 1.9]{Derksen17}), using
Stillman's conjecture and Draisma's theorem, the same authors establish
a generalization of Stillman's conjecture to ideal invariants that
are upper semicontinuous in flat families and preserved under adding a
variable to the polynomial ring. We have not pursued the question to
what extent (an algorithmic version of) this generalisation also follows from our
Theorem~\ref{thm:Stillman}. 

\subsection*{Organization}

This paper is organized as follows. In Section~\ref{sec:GB} we
prove Theorem~\ref{thm:GBExists} using work from \cite{Erman18}. In
Section~\ref{sec:Buchberger} we use this existence result to prove
that a version of Buchberger's algorithm for eventually invariant
series terminates; this yields Theorem~\ref{thm:GBComputable}. In
Section~\ref{sec:PolFunctor} we review topological Noetherianity of a
specific polynomial functor, which follows from~\cite{Draisma17}. Finally,
in Section~\ref{sec:Stillman} we derive Theorem~\ref{thm:Stillman} from
Theorem~\ref{thm:GBComputable} and Draisma's theorem.

\section{The existence of finite Gr\"obner bases}
\label{sec:GB}

We will use two results from~\cite{Erman18}, the first of which is the following.

\begin{thm}[Theorem~1.2 from \cite{Erman18}] \label{thm:12}
If $L$ is perfect, then $R_L$ contains an (uncountable) set of homogeneous
elements $\{g_j:j \in J\}$ such that the unique $L$-algebra homomorphism
$L[(x_j)_{j \in J}] \to R_L$ sending $x_j$ to $g_j$ is an $L$-algebra
isomorphism.
\end{thm}

For each $n \in \ZZo$ we write $R_L^{(n)}:=L[x_1,\ldots,x_n]$. There is
a natural $L$-algebra homomorphism $R_L \to R_L^{(n)}, f \mapsto f^{(n)}$
that retains only the terms involving only the variables $x_1,\ldots,x_n$.
We may think of a degree-at-most-$d$ element of $R_L$ as a sequence
$(f^{(0)},f^{(1)},\ldots)$ in which each $f^{(n)}$ is a polynomial in
$R_L^{(n)}$ of degree at most $d$ such that $f^{(n)}$ is the image of
$f^{(n+1)}$ under discarding all terms divisible by $x_{n+1}$. Conversely,
$R_L^{(n)}$ is an $L$-subalgebra of $R_L$.  Observe that, for any $f
\in R_L$ and $n \in \ZZo$, the image $(\lmon(f))^{(n)}$ is either zero
or equal to $\lmon(f^{(n)})$ in the grevlex order on $L[x_1,\ldots,x_n]$.

\begin{thm}[Theorem 5.4 from \cite{Erman18}]\label{thm:54}
A sequence $g_1,\ldots,g_l \in R_L$ of homogeneous elements is a
regular sequence in $R_L$ if and only if $g_1^{(n)},\ldots,g_l^{(n)}$
is a regular sequence in $R_L^{(n)}$ for all $n \gg 0$.
\end{thm}

The following lemma is straightforward from \cite[Section 5]{Erman18}, but we
include its proof using the two results above.

\begin{lm} \label{lm:Syzygies}
Let $f_1,\ldots,f_k \in R_L$. Then the natural map between
first syzygies 
\[ \Syz_{R_L^{(n+1)}}(f_1^{(n+1)},\ldots,f_k^{(n+1)}) \to 
\Syz_{R_L^{(n)}}(f_1^{(n)},\ldots,f_k^{(n)}) \]
is surjective for all $n \gg 0$.
\end{lm}

\begin{proof}
This surjectivity is not affected by enlarging the field, so we
may assume that $L$ is perfect. By Theorem~\ref{thm:12},
there exist homogeneous
$g_1,\ldots,g_l \in R_L$ such that $f_1,\ldots,f_k \in L[g_1,\ldots,g_l]$
and such that $g_1,\ldots,g_l$ are part of a system of variables
for the polynomial ring $R_L$. In particular, they are a regular
sequence in $R_L$, and hence by Theorem~\ref{thm:54} the polynomials
$g_1^{(n)},\ldots,g_l^{(n)}$ are a regular sequence in $R_L^{(n)}$
for $n \gg 0$. We draw two conclusions from this. First, for $n \gg
0$, $g_1^{(n)},\ldots,g_l^{(n)}$ are algebraically independent over
$L$, $f_1^{(n)},\ldots,f_k^{(n)}$ are elements of the polynomial
ring $A^{(n)}:=L[g_1^{(n)},\ldots,g_l^{(n)}]$, and 
\[ \Syz_{A^{(n+1)}}(f_1^{(n+1)},\ldots,f_k^{(n+1)}) 
\to 
\Syz_{A^{(n)}}(f_1^{(n)},\ldots,f_k^{(n)})  \]
is a bijection. Second, still for $n \gg 0$, $R^{(n)}$ is a
free module over $A^{(n)}$. Therefore,
$\Syz_{A^{(n)}}(f_1^{(n)},\ldots,f_k^{(n)}) \subseteq
(A^{(n)})^k$ generates
$\Syz_{R^{(n)}}(f_1^{(n)},\ldots,f_k^{(n)}) \subseteq
(R^{(n)})^k$ as an
$R^{(n)}$-module. Combining these two statements we find the
surjectivity claimed in the lemma.
\end{proof}

\begin{proof}[Proof of Theorem~\ref{thm:GBExists}]
Let $f_1,\ldots,f_k \in R_L$ be nonzero, homogeneous, and let
$n \in \ZZo$. Set $I:=\langle f_1,\ldots,f_k \rangle \subseteq
R_L$.  Consider a monomial $u \in \lmon(I) \cap R^{(n+1)}$ divisible by
$x_{n+1}$. There exist homogeneous $a_1,\ldots,a_k \in R^{(n)}$ with
$\deg(a_i)=\deg(u)-\deg(f_i)$ and homogeneous $b_1,\ldots,b_k \in
R^{(n+1)}$ with $\deg(b_i)=deg(u)-deg(f_i)-1$ such that
\[ u = \lmon((a_1+b_1 x_{n+1})f_1^{(n+1)} + \cdots + (a_k +
b_k x_{n+1})f_k^{(n+1)}). \]
Now $(a_1,\ldots,a_k)\in \Syz_{R^{(n)}}
(f_1^{(n)},\ldots,f_k^{(n)})$---otherwise, the right-hand side would equal
$\lmon(\sum_i a_i f_i^{(n)})$, which is not divisible by $x_{n+1}$. By
Lemma~\ref{lm:Syzygies}, if $n \gg 0$, the syzygy $(a_1,\ldots,a_k)$ can
be lifted to a syzygy $(c_1,\ldots,c_k) \in
\Syz_{R^{(n+1)}_L}(f^{(n+1)}_1,\ldots,f^{(n+1)}_k)$.  Write $c_i=a_i
+ x_{n+1} b'_i $ for each $i$.  Then
\[ u = \lmon((b_1-b'_1)x_{n+1}f_1^{(n+1)} + \cdots +
(b_k-b'_k)x_{n+1}f^{(n+1)}), \]
but then we see that $u/x_{n+1} \in \lmon(I)$. Hence for $n \gg 0$,
$\lmon(I)$ does not contain minimal generators divisible by $x_{n+1}$. It
follows that for such an $n$, $\lmon(I)$ is generated by any finite
generating list $m_1,\ldots,m_t$ of $\lmon(I^{(n)})$. Now $h_1,\ldots,h_t
\in I$ such that $\lmon(h_i)=m_i$ form a Gr\"obner basis of $I$.
\end{proof}

\section*{Buchberger's algorithm for grevlex series}

To turn Theorem~\ref{thm:GBExists} into an algorithm, we
derive a version of Buchberger's algorithm.

\begin{lm}[Division with remainder] \label{lm:Division}
Let $f_1,\ldots,f_k \in R_L$ be monic and $h \in R_L$. Then there
exist $q_1,\ldots,q_k \in R_L$ such that $\lmon(q_i f_i) \leq \lmon(h)$
for all $i$ and such that no term of the remainder $h-\sum_i q_i f_i$
is divisible by any $\lmon(f_i)$.
\end{lm}

In particular, if $f_1,\ldots,f_k$ is a Gr\"obner basis of the ideal
that they generate, then the remainder must be zero.

\begin{proof}
Initialize $r:=h$ and $q_i:=0$ for all $i$. While some term of $r$
is divisible by some $\lmon(f_i)$, pick the grevlex-largest such term
$c x^\alpha$ in $r$, subtract $c (x^{\alpha}/\lmon(f_i)) f_i$ from $r$
and add $c x^{\alpha}/\lmon(f_i)$ to $q_i$. This does not change terms in
$q_i$ larger than the term just added, and hence in the product topology
on $R_L^k$ the vector $q$ converges a solution vector $q$ as desired.
\end{proof}

\begin{algorithm}
\begin{algorithmic}
\Function{SeriesBuchberger}{$f_1,\ldots,f_k$}
  \State {\bf assume} $f_1,\ldots,f_k \in R_L$ homogeneous grevlex series.
  \State $m:=1$; $B:=\emptyset$ (the basis);
  $Q:=\{f_1,\ldots,f_k\}$ (the queue);
  \While{$Q \neq \emptyset$} 
    \While{$Q$ contains an $f$ with $f^{(m)} \neq 0$}
      \State $Q:=Q \setminus \{f\}$; 
      \State $f:=f/\lc(f)$; 
      \State $B:=B \cup \{f\}$;
      \For{$h \in B \setminus\{f\}$} 
        \State $\gamma:=\lcm(\lmon(h),\lmon(f))$;
        \State $s:=(\gamma/\lmon(h))h-(\gamma/\lmon(f))f$;
        \State compute a remainder $r$ of $s$ after division by $B$;
        \If{$r \neq 0$} 
          \State $Q:=Q \cup \{r\}$; 
        \EndIf;
      \EndFor;
    \EndWhile;
    \State $m:=m+1$;
  \EndWhile;
\State \Return $B$;
\EndFunction
\end{algorithmic}
\label{alg:SeriesBuchberger}
\end{algorithm}

\begin{prop} \label{prop:SeriesBuchberger}
Assuming an implementation for addition, multiplication, and
division with remainder of grevlex series, {\sc SeriesBuchberger} on
page~\pageref{alg:SeriesBuchberger} terminates after a finite number of
steps and outputs a Gr\"obner basis of the ideal generated by the series
in the input.
\end{prop}

\begin{proof}
Fix any natural number $n$.  The loops with $m$ ranging from
$1$ to $n$ really compute a Gr\"obner basis for $I:=\langle
f_1^{(n)},\ldots,f_k^{(n)} \rangle$ while dragging the tails of the
series along. In particular, these $n$ loops terminate.  If an element
is added to the queue $Q$ in the $(n+1)$st run of the loop, then this
implies that $\lmon(I) \cap R_L^{(n)}$ does not generate $\lmon(I)$. By
Theorem~\ref{thm:GBExists}, this cannot happen infinitely often, so the
algorithm terminates. That the output is, indeed, a Gr\"obner basis,
follows from the ordinary Buchberger criterion.
\end{proof}

\section{Buchberger's algorithm for eventually invariant
series} \label{sec:Buchberger}

Recall that $\mathrm{S}_{>n_0}$ acts on $L$, on variables, and on
$R_L$. Given representations $\hat{f}_1,\ldots,\hat{f}_k$ of eventually
invariant $f_1,\ldots,f_k \in R_L$, we want to compute the representation
of an eventually invariant Gr\"obner basis of $I:=\langle f_1,\ldots,f_k
\rangle$.

The first ingredient in our variant of Buchberger's algorithm is an
analogue of Lemma~\ref{lm:Division}.

\begin{lm}[Division with remainder on representations.] \label{lm:Division2}
Let $f_1,\ldots,f_k \in R_L$ be monic and $h \in R_L$. Assume that
$h,f_1,\ldots,f_k,\lmon(f_1),\ldots,\lmon(f_k)$ are invariant under
$\Sgn$.  Then $q_1,\ldots,q_k$ and $r$ from Lemma~\ref{lm:Division}
can be chosen $\Sgn$-invariant, and the representations
$\hat{q}_1,\ldots,\hat{q}_k,\hat{r}$ can be effectively computed from
$\hat{h},\hat{f}_1,\ldots,\hat{f}_r$.
\end{lm}

\begin{proof}
Set $r:=h$. While some term of $r$ is divisible by some $\lmon(f_i)$,
pick the grevlex-largest such term $c x^\alpha$ in $r$, let
$x^{\alpha_1},x^{\alpha_2},\ldots$ be the (countably infinite) orbit of
$x^{\alpha}$ under $\Sgn$ and for each $i$ let $c_i$ be the coefficient of
$x^{\alpha_i}$ in $r$. Since $r$ is $\Sgn$-invariant, so is $a:=\sum_i
c_i x^{\alpha_i}$.  Moreover, as $\lmon(f_i)$ is $\Sgn$-invariant,
$a$ is divisible by $\lmon(f_i)$. Replace $r$ by $r-(a/\lmon(f_i))f_i$
and $q_i$ by $q_i+(a/\lmon(f_i))$; each of these are $\Sgn$-invariant.
This does not effect the terms of $r$ larger than $x^{\alpha}$ and gets
rid of this particular term. In this process, $r$ and the $q_i$ remain
$\Sgn$-invariant and converge to series as in Lemma~\ref{lm:Division}.

For effectiveness, we need to be able to compute the representation
of $r-(a/\lmon(f_i))f_i$ from $\hat{r},\hat{a}=c x^\alpha,$ and
$\hat{f}_i$. The representation depends linearly on the series, so
it suffices to have a procedure for computing the representation of
a product. The function {\sc Product} does just that---it uses that
no monomial of degree $e$ that is grevlex-maximal in its $\Sgn$-orbit
contains any of the variables $x_{n+e+1},x_{n+e+2},\ldots$.
\end{proof}

\begin{algorithm}
\begin{algorithmic}
\Function{Product}{$n,\hat{f},\hat{h}$}
  \State {\bf input:} $n$-representations $\hat{f},\hat{h}$ of
  $\Sgn$-invariant series $f,h$.
  \State {\bf output:} $n$-representation $\widehat{fh}$ of
  $fh$.
  \State $e:=\deg{f}+\deg{h}$;
  \State compute the truncations $f^{(n+e)},h^{(n+e)}$ from
  $\hat{f},\hat{h}$;
  \State $u:=f^{(n+e)} h^{(n+e)}$; 
  \State remove all terms in $u$ not grevlex-maximal in
  their $\Sgn$-orbit;
  \State \Return $u$;
\EndFunction
\end{algorithmic}
\end{algorithm}

\begin{algorithm}
\begin{algorithmic}
\Function{Remainder}{$n,\hat{h},\{\hat{f}_1,\ldots,\hat{f_k}\}$}
\State
{\bf input:} $n$-representations $\hat{h},\hat{f}_1,\ldots,\hat{f}_k$
of $\Sgn$-invariant series, with
$\hat{f}_i$ monic 
\State \quad and $\lmon(\hat{f}_i)$ $\Sgn$-invariant;
\State 
{\bf output:} the $n$-representation of a remainder of $h$
after division by $f_1,\ldots,f_k$.
\State 
{\bf assume} $\lmon(\hat{f}_1),\ldots,\lmon(\hat{f}_k)$ are
$\Sgn$-invariant.
\State $\hat{r}:=\hat{h}$;
\While{$\hat{r}$ contains a term $c x^\alpha$ divisible by
some $\lmon(\hat{f_i})$}
	\State
	$\hat{r}:=\hat{r}-${\sc
	Product}$(cx^\alpha/\lmon(\hat{f}_i),x \hat{f}_i)$;
\EndWhile;
\State \Return $\hat{r}$;
\EndFunction
\end{algorithmic}
\end{algorithm}

The next ingredient is $S$-series: if $f,g$ are monic
$\Sgn$-invariant series whose leading monomials are also
$\Sgn$-invariant, and $x^\gamma=\lcm(\lmon(f),\lmon(g))$, then we
set $S(f,g):=(x^\gamma/\lmon(f))f-(x^\gamma/\lmon(g))g$.  We note
that $S(f,g)$ is also $\Sgn$-invariant, and the $n$-representation of
$S(f,g)$ can be computed from the $n$-representations $\hat{f},\hat{g}$,
as follows.

\begin{algorithm}
\begin{algorithmic}
\Function{S}{$n,\hat{f},\hat{g}$}
	\State {\bf input:}  
	$n$-representations $\hat{f},\hat{g}$ of monic $\Sgn$-invariant 
series with $\lmon(\hat{f}),\lmon(\hat{g})$
	\State \quad $\Sgn$-invariant.
	\State {\bf output:} $n$-representation of $S(f,g)$.
	\State $x^\gamma:=\lcm(\lmon(\hat{f}),\lmon(\hat{g}))$;
	\State
	$\hat{s}:=(x^\gamma/\lmon(\hat{f}))\hat{f}-(x^\gamma/\lmon(\hat{g}))\hat{g}$;
	\State \Return $\hat{s}$;
\EndFunction
\end{algorithmic}
\end{algorithm}

From the $n$-representation $\hat{f}=\sum_{i=1}^s c_{\alpha_i}
x^{\alpha_i}$ of an $\Sgn$-invariant grevlex series one can
compute the $m$-representation $\tilde{f}$ with $m>n$ as follows. For each
$i=1,\ldots,s$, the group $\mathrm{S}_{>m}$ has only finitely many
orbits on $\Sgn x^{\alpha_i}$. Let
$x^{\beta_{i1}},\ldots,x^{\beta_{is_i}}$
be the grevlex-maximal representatives of these orbits, and let
$\pi_{i1},\ldots,\pi_{is_i} \in \mathrm{S}_{>m}$ be such that
$\pi_{ij} x^{\alpha_i}=x^{\beta_{ij}}$. Then define
\[
\tilde{f}:=\sum_{i=1}^s \sum_{j=1}^{s_i}
\pi_{ij}(c_{\alpha_i}) x^{\beta_{ij}}. 
\]
We call $\tilde{f}$ the $m$-{\em expansion} of $\hat{f}$. So we may
freely increase $n$ when desirable; we will use this to ensure that the
leading monomial of $f$ is $\Sgn$-invariant.

\begin{algorithm}
\begin{algorithmic}
\Function{SymmetricBuchberger}{$n,\hat{f}_1,\ldots,\hat{f}_k$}
\State {\bf input:} $n$-representations $\hat{f}_1,\ldots,\hat{f}_k$ 
$\Sgn$-invariant series.
\State {\bf output:} the $m$-representation of a Gr\"obner
basis of $\langle f_1,\ldots,f_k \rangle$ for some $m \geq n$.
  \State $m:=n$; $\hat{B}:=\emptyset$ (the basis);
  $Q:=\{\hat{f}_1,\ldots,\hat{f}_k\}$ (the queue);
  \While{$Q \neq \emptyset$} 
    \While{$Q$ contains an $\hat{f}$ with $\hat{f}^{(m)} \neq 0$}
      \State $Q:=Q \setminus \{\hat{f}\}$; 
      \State $\hat{f}:=\hat{f}/\lc(\hat{f})$; 
      \State $\hat{B}:=\hat{B} \cup \{\hat{f}\}$;
      \For{$\hat{h} \in \hat{B} \setminus
      \{\hat{f}\}$} 
      \State
      $\hat{r}:=${\sc Remainder}$(m,$S$(m,\hat{h},\hat{f}),\hat{B})$;
      \EndFor;
      \If{$\hat{r} \neq 0$} 
        \State $Q:=Q \cup \{\hat{r}\}$; 
      \EndIf;
    \EndWhile;
    \State Replace $\hat{B}$ and $\hat{Q}$ by their $(m+1)$-expansions;
    \State $m:=m+1$;
  \EndWhile;
\State \Return $\hat{B}$;
\EndFunction
\end{algorithmic}
\label{alg:SymmetricBuchberger}
\end{algorithm}

\begin{proof}[Proof of Theorem~\ref{thm:GBComputable}.]
The algorithm {\sc SymmetricBuchberger} above called with arguments
$(n,\hat{f_1},\ldots,\hat{f_k})$ performs the same operations as the
algorithm {\sc SeriesBuchberger} on page~\pageref{alg:SeriesBuchberger} on
input $(f_1,\ldots,f_k)$, except that it works with finite data structures
capturing the series $f_i$. Hence Proposition~\ref{prop:SeriesBuchberger}
implies both the termination and the fact that the output
of {\sc SymmetricBuchberger} is the representation of a Gr\"obner basis of
$\langle f_1,\ldots,f_k \rangle$.
\end{proof}

\begin{re}
Should one consider implementing {\sc SymmetricBuchberger}, it may be
practical to allow series to have $m$-representations with varying values
of $m$, as opposed to the uniform $m$ for every iteration of the outer
loop above.

In order to perform binary operations, i.e., additions and
multiplications, representations would then need to be expanded to a
matching value of $m$. Furthermore, to ensure termination, the order
of S-pairs needs to ensure that each leading monomial is eventually
encountered. In {\sc SymmetricBuchberger}, this is done by increasing
$m$ only after all the leading monomials in $x_1,\ldots,x_m$ have been
collected.
\end{re}

\section{A polynomial functor} \label{sec:PolFunctor}

Let $K$ be an infinite field. Let $\GL_n(K)$ act on the space $K^n$
with basis $x_1,\ldots,x_n$ by left multiplication, and for each $d \in
\ZZo$ on the $d$-th symmetric power $S^d K^n$ in the natural manner. Fix
$d_1,\ldots,d_k \in \ZZo$ and set
\[ P^{(n)}(K):= S^{d_1}K^n \oplus \cdots \oplus S^{d_k}K^n, \]
the space of tuples of forms of degrees $d_1,\ldots,d_k$ in $n$ variables. Now
define $P(K):=\lim_{\ot n} P^{(n)}(K)$, the projective limit along
the maps $P^{(n+1)}(K) \to P^{(n)}(K)$ coming from the projection
$K^{n+1} \to K^n$ forgetting the last coordinate. The map
$P^{(n+1)}(K) \to P^{(n)}(K)$ 
is $\GL_n(K)$-equivariant if we think of $\GL_n(K)$ as embedded in
$\GL_{n+1}(K)$ via the map $g \mapsto \diag(g,1)$, and hence $P(K)$ is
a module for the group $\GL_\infty(K):=\bigcup_n \GL_n(K)$.
The space $P(K)$
is the subspace of $(R_K)^k$ consisting of all tuples where the $i$-th
element is homogeneous of degree $d_i$ for each $i \in [k]$.

Dually, let $V:=\lim_{n \to } (P^{(n)}(K))^*$. Then $V$ is a
countable-dimensional space, $P(K)$ is canonically isomorphic to $V^*$,
and hence $K[P]:=SV$, the symmetric algebra on $V$, serves as a coordinate
ring of $P(K)$ in that the set of $K$-algebra homomorphisms $K[P] \to K$
is canonically identified with $P(K)$. We equip $P(K)$ with the Zariski
topology in which closed subsets are characterized by polynomial equations
from $K[P]$. Also $V$ and $K[P]$ are modules for $\GL_\infty(K)$. The
following is an instance of a general result on polynomial functors
from \cite{Draisma17}.

\begin{thm} \label{thm:Noetherianity}
Let $K$ be an infinite field, and fix integers $d_1,\ldots,d_N$. Then
any chain $P(K) \supseteq X_1 \supseteq \cdots$ of
$\GL_\infty(K)$-stable
Zariski-closed subsets stabilizes eventually. Equivalently, any sequence
$a_1,a_2,a_3,\ldots$ in $K[P]$ has the property that for $t \gg 0$ we have
\[ a_t \in \sqrt{\left\langle \bigcup_{i=1}^{t-1}
\GL_\infty(K) a_i
\right\rangle}. \]
\end{thm}

\begin{re} \label{re:Finite}
Two comments are in order. First, the implication $\Rightarrow$ between
the two statements in the theorem follows from the Nullstellensatz, since
the first sentence also holds for any algebraic closure 
of $K$. Second, each $a_i$ is an element of $K[P^{(n_i)}]$ for some
finite $n_i$. If $n \geq \max_{i\in [t]} n_i$, then the property of $a_t$
above is equivalent to
\[ a_t \in \sqrt{\left\langle \bigcup_{i=1}^{t-1} \GL_n(K)
a_i \right\rangle}, \]
where we have replaced $\infty$ by $n$. 
\end{re}

\section{Finitely many generic initial ideals}
\label{sec:Stillman}

We now prepare for the proof of Theorem~\ref{thm:Stillman}. For $i=1,\ldots,k$
let $f_i$ be the homogeneous degree-$d_i$ series
\[ f_i=\sum_{|\alpha|=d_i}c_{i,\alpha} x^\alpha \]
whose coefficients live in the polynomial ring
\[ A=\ZZ[c_{i,\alpha} \mid i \in [k], \alpha \in \ZZo^\NN, |\alpha|=d_i] \]
in which the $c_{i,\alpha}$ are variables. We note that if
$K$ is a field, then $K \otimes A$ is the coordinate ring
$K[P]$ of the space $P(K)$ introduced in
Section~\ref{sec:PolFunctor}.

On $A$ acts $\Sinf$ via ring automorphisms determined by $\pi
c_{i,\alpha}=c_{i,\alpha \circ \pi^{-1}}$, and each $f_i$ is
$\Sinf$-invariant. In the $0$-representation $\hat{f}_i$ of $f_i$, we have
\[ \hat{f}_i=\sum_{|\alpha|=d_i, \alpha(1)\geq \alpha(2) \geq \ldots}
c_{i,\alpha} x^\alpha, \]
a polynomial with as many terms as there are partitions of $d_i$.
Write $A^{(n)}$ for the subring of $A$ generated by those $c_{i,\alpha}$
such that $\forall m > n: \alpha(m)=0$.

Let $g$ be an $n \times n$-matrix of variables. Replacing, in $f_i$,
each $x_h$ with $h \leq n$ by $\sum_j g_{hj} x_j$ and each $x_h$ with
$h>n$ by $x_h$ yields a series $gf_i$ in the $x_h$ whose coefficients
are polynomials that are linear in the $c_{i,\alpha}$ and homogeneous
of degree $d_i$ in the $g_{hj}$. We use the formal notation $g^{-1}
c_{i,\alpha}$ for the coefficient of $x^\alpha$ in $gf_i$. This notation
is chosen so that if we specialize $g$ to be the matrix of a permutation
$\pi \in \Sn$, then $g^{-1} c_{i,\alpha}$ specializes to $\pi^{-1}
c_{i,\alpha}$ in the $\Sn$-action above.  For a polynomial
$r=r(c) \in A$
(in the $c_{i,\alpha}$ with varying $i$ and $\alpha$) write $g^{-1}
r \in \ZZ[g_{hj} \mid h,j \in [n]] \otimes_\ZZ A$ for the polynomial obtained by replacing each $c_{i,\alpha}$
with $g^{-1} c_{i,\alpha}$. Regarding $g^{-1} r$ as a polynomial in
the entries $g_{hj}$ whose coefficients are in $A$, we write $E_n(r)
\subseteq A$ for the set of all nonzero coefficients. It is easy to see
that if $r \in A^{(n)}$, then also $E_n(r) \subseteq A^{(n)}$.
The following easy lemma explains the significance of this construction.

\begin{lm} \label{lm:Orbit}
If $K$ is an infinite field, then the $K$-span of the orbit of $1 \otimes r
\in K \otimes_\ZZ A$ under $\GL_n(K)$ equals the $K$-span of $1 \otimes
E_n(r(c))$. \hfill $\square$
\end{lm}

\begin{algorithm}
\begin{algorithmic}
\Procedure{Stillman}{$n,\hat{B},\hat{Q},Y,Z,N$}
  \State $m:=n$;
  \State $Q:=Q \setminus \{0\}$;
  \While{$Y\neq\emptyset$ and $\hat{Q} \neq \emptyset$} 
    \While{$Y\neq\emptyset$ and $\hat{Q}$ contains an $\hat{f}$ with $\hat{f}^{(m)} \neq 0$}
      \State $\hat{Q}:=\hat{Q} \setminus \{\hat{f}\}$;
      \State $b:=\lc(\hat{f})\in A[N^{-1}]$;
      \State $a:=$numerator$(b) \in A$;
      \State $Y_1:=\left\{(p) \in Y: 
      	a \in \sqrt{\langle \bigcup_{r \in Z} E_m(r)
          \rangle_{\FF_p  \otimes A^{(m)}[N^{-1}]}}\right\}$;
      \State {\sc Stillman}($m,\hat{B},\hat{Q}\cup
		\{\hat{f}-\lt(\hat{f})\},Y_1,Z,N$); \hfill
		(I)
      \State $Y:=Y \setminus Y_1$;
        \State $Y_2:=\left\{(p) \in Y:
          1 \not \in \langle \bigcup_{r \in Z \cup \{a\}} E_m(r)
          \rangle_{\FF_p \otimes A^{(m)}[N^{-1}]}\right\}$;
        \State {\sc Stillman}($m,\hat{B},\hat{Q}\cup\{\hat{f}-\lt(\hat{f})\},Y_2,Z\cup\{a\},N$);
        \hfill (II)
        \State $\hat{f}:=\hat{f}/b$; 
        \State $N:=N \cup \{a\}$;
        \State $\hat{B}:=\hat{B} \cup \{\hat{f}\}$;
        \For{$\hat{h} \in \hat{B} \setminus\{\hat{f}\}$} 
          \State $\hat{r}:=${\sc Remainder}$(m,$S$(m,\hat{h},\hat{f}),B)$;
          \If{$\hat{r} \neq 0$} 
            \State $\hat{Q}:=\hat{Q} \cup \{\hat{r}\}$; 
          \EndIf;
        \EndFor;
    \EndWhile;
    \State Replace $B$ and $\hat{Q}$ by their $(m+1)$-expansions;
    \State $m:=m+1$;
  \EndWhile;
  \If{$Y\neq\emptyset$}\State {\bf print} $\lmon(\hat{B})$; \EndIf;
  %\State {\bf print} $(Y,Z,N)$;
\EndProcedure
\end{algorithmic}
\label{alg:Stillman}
\end{algorithm}

\begin{proof}[Proof of Theorem~\ref{thm:Stillman}]
In the recursive variant {\sc Stillman} of {\sc SymmetricBuchberger} on
page~\pageref{alg:Stillman}, we write $\FF_p$, where $p$ is either zero or
a prime, for the prime field of characteristic $p$.  
%The command {\bf exit procedure} means an exit back to the run of the procedure from which the current run was called, not a termination of the computation altogether.
We prove that {\sc Stillman} terminates on input
\[ (0,\emptyset,\{\hat{f}_1,\ldots,\hat{f}_d\},\Spec(\ZZ),\emptyset,\emptyset)
\]
and that it prints out the sets $S_i$ as in the theorem.

First we clarify the role of the variables. The symbols
$m,\hat{B},\hat{Q}$ carry the same meaning as in {\sc
SymmetricBuchberger}.  The meaning of $Z$ and $N$, finite subsets of
$A$, is that of vanishing and nonvanishing elements, respectively, at
the current run of the algorithm.  While $Z$ stays constant throughout
the run ($Z$ is extended only when recursive calls are made), $N$ is
augmented as it accumulates elements due to presumed nonvanishing of
the leading coefficients.

The current run considers only primes in the set
$Y \subseteq \Spec(\ZZ)$. Furthermore, it deals with the specializations
of the truncations $f_1^{(m)},\ldots,f_k^{(m)}$ with coefficients in
\[ \bar{A}^{(m)}:=A^{(m)}[N^{-1}]/\sqrt{\left\langle\bigcup_{r
\in Z} E_m(r)\right\rangle}. \]
We discuss the computations of $Y_1,Y_2$. 

For $Y_1$, one starts
running the ordinary Buchberger algorithm on the ideal in the localization
$A^{(m)}[N^{-1}][t]$ generated by $\bigcup_{r \in Z} E_m(r)$ and $ta-1$
(Rabinowitsch' trick), where $t$ is an auxiliary variable.  Whenever an
integer leading coefficient is divisible by a nonzero prime $(p)$ in $Y$,
the algorithm branches into a branch where multiples of $p$ are zero and a
branch where $p$ is invertible. Assuming that $Y$ is constructible to begin
with, each leaf of this finite tree yields a constructible set of primes
leading to that leaf, and $Y_1$ is the union of the primes corresponding
to leaves where the aforementioned ideal contains $1$.  

A similar algorithm
is used to compute $Y_2$. Since we start with $Y=\Spec \ZZ$, it follows
that in any of the further calls of {\sc Stillman} the set $Y$
is constructible. In other words, $Y$ is either a finite set of nonzero
primes in $\Spec(\ZZ)$ or a cofinite set in $\Spec(\ZZ)$ containing $(0)$.

Furthermore, in each run of {\sc Stillman}, for each $(p) \in Y$,
the algebra $\FF_p \otimes \bar{A}^{(m)}$ has $0 \neq 1$. This is true
at the initial call, it remains true in call (I) since $m,Z,N$ do not
change, and it remains true in call (II) since we explicitly test for
this condition. Furthermore, it remains true later in the loop, since
there we have already removed from $Y$ the primes in $Y_1$, which are
those where inverting $a$ would cause the algebra to collapse.

Let $T$ be the rooted tree whose vertices are the runs of {\sc Stillman}
and whose edges are labelled (I) or (II) according to which call in the
algorithm leads from one run to the other. We claim that every path in
$T$ away from the root is finite. Indeed, consider an infinite
path $\gamma$ in $T$. The argument $Y$ remains nonempty and weakly
decreases along $\gamma$ and since it is locally closed in $\Spec(\ZZ)$,
so there exists a prime $(p_0) \in \Spec \ZZ$ that is in the intersection
of all the arguments $Y$ along $\gamma$.

If infinitely many edges in $\gamma$ are labelled (II), then 
$Z$ records $a_1,a_2,a_3,\ldots$ with $a_i \in A^{(m_i)}$ and $m_1 \leq
m_2 \leq \ldots$ and
\[ 
	a_i \not \in \sqrt{\left\langle E_{n_i}(a_1) \cup \cdots
	\cup E_{n_i}(a_{i-1}) \right\rangle_{\FF_{p_0} \otimes A^{(m_i)}}
	}
	\text{ for all } i=1,2,3,\ldots  
\]
By Lemma~\ref{lm:Orbit} and Remark~\ref{re:Finite} this contradicts
the Noetherianity of $K[P]$ over any infinite field $K$ of
characteristic $p_0$ (Theorem~\ref{thm:Noetherianity}).

Hence only finitely many edges in $\gamma$ are labelled (II). We analyse
the computation along $\gamma$ beyond the last edge $e$ labelled (II). Let
$m_\infty \in \NN \cup \{\infty\}$ be the supremum of the values of $m$
along $\gamma$, let $Z_0$ be the (fixed) value of $Z$ along $\gamma$
from $e$ onwards, and let $N_\infty$ the union of all $N$'s seen along
$\gamma$. Define
\[ \tilde{A}:=\FF_{p_0} \otimes
A[N_\infty^{-1}]/\sqrt{\left\langle \bigcup_{m \leq m_\infty, r \in Z}
E_m(r) \right\rangle}. \]
By construction, $1 \neq 0$ in $\tilde{A}$, hence there exists an
epimorphism from $\tilde{A}$ to some field $L$ of characteristic $p_0$.
Then the call of {\sc SeriesBuchberger} with input the images of the
$f_i$ in $R_L$ performs the same operations as the algorithm {\sc
Stillman} along $\gamma$. Since the former algorithm terminates by
Proposition~\ref{prop:SeriesBuchberger}, so does the latter.

We conclude that $T$ is finite. Let $K$ be an infinite field, $n$ a
natural number, and let $f'_1,\ldots,f'_k$ be homogeneous polynomials
in $K[x_1,\ldots,x_n]$ of degrees $d_1,\ldots,d_k$, respectively.
We claim that, at the leaf of some path $\gamma$ in $T$ away from
the root, generators for the generic initial ideal of $\langle
f'_1,\ldots,f'_k \rangle$ are printed. To see this, let $g$
be an $n \times n$-matrix of variables, set $L:=K((g_{hj})_{h,j})$,
and consider the ideal $J$ in $L[x_1,\ldots,x_n]$ generated by the
polynomials $gf'_1,\ldots,gf'_k$ obtained by replacing $x_h$ in each
$f'_i$ by $\sum_h g_{hj} x_j$. Then the generic initial ideal of $\langle
f'_1,\ldots,f'_k \rangle \subseteq K[x_1,\ldots,x_n]$ equals the initial
ideal of $J$, and the latter is computed by {\sc
Buchberger} (or {\sc SymmetricBuchberger}) on
input ($n$ and) $gf'_1,\ldots,gf'_k$. To find $\gamma$, proceed as follows:
whenever $a \in A$ is defined as the numerator of a leading
coefficient of $\hat{f}$, check if under the specialization
$f_i \mapsto gf'_i$ the element $a$ specializes in $K$ to zero
or to a nonzero element. If $a$ specializes to zero, then follow call
(I) or call (II) according as $(\cha K) \in Y_1$ or not. If $a$ does
not specialize to zero, then follow neither of these calls and continue
with the loop.  Along this $\gamma$, {\sc Stillman} performs the same
operations as {\sc SeriesBuchberger}, and hence terminates with the
generic initial ideal of $\langle f'_1,\ldots,f'_k \rangle$.
\end{proof}

\begin{re}
Apart from printing $\lmon(\hat{B})$ at each leaf of $T$ we may also
print $Y,Z,N$, which together describe a locally closed stratum 
of $P(K)$, for any infinite $K$ with $(\cha K) \in Y$, consisting of
$k$-tuples with generic initial ideal generated by $\lmon(\hat{B})$.
\end{re}

\bigskip

{\bf Acknowledgements.}
The authors thank Institut Mittag-Leffler for the hospitality during their stay in the early 2018. 

The first author thanks Diane Maclagan, who introduced
Stillman's conjecture to him many years ago and suggested
that it might be related to Noetherianity up to symmetry.

The third author thanks MSRI for the Fall of 2017 stay, during which ``The Fellowship of the Ring'' and Craig Huneke rekindled his interest in Stillman's conjecture. Last but the most, the thanks go to Gennady Lyubeznik and Mike Stillman. When they both were 40 years young, Gennady suggested a problem that involved a certain constructible stratification and Mike helped solving it via parametric Gr\"obner bases. The intuition gained 20 years ago was instrumental in the present work.

%\bibliographystyle{alpha}
%\bibliography{diffeq,draismapreprint}

\end{document}